\documentclass[psreqno,reqno,11pt]{amsart}
\usepackage{eurosym}
\usepackage{amssymb}
\usepackage{tikz}
\usepackage{amsmath}
\usepackage{tikz}
\usepackage{pgflibraryplotmarks}
\usepackage{wasysym}
\usepackage{stmaryrd}
\usepackage[english]{babel}

\usepackage{hyperref}

\theoremstyle{plain}
\newtheorem{theorem}{Theorem}[section]
\newtheorem{proposition}[theorem]{Proposition}
\newtheorem{lemma}[theorem]{Lemma}
\newtheorem{notation}[theorem]{Notation}
\newtheorem{corollary}[theorem]{Corollary}

\theoremstyle{definition}
\newtheorem{definition}[theorem]{Definition}
\newtheorem{example}[theorem]{Example}
\newtheorem{remark}[theorem]{Remark}

\begin{document}
	\def\sect#1{\section*{\leftline{\large\bf #1}}}
	\def\th#1{\noindent{\bf #1}\bgroup\it}
	\def\endth{\egroup\par}
	
	\title[Total orderization invariant maps on distributive lattices]
	{Total orderization invariant maps on distributive lattices}
	\author{Christopher Michael Schwanke}
	\address{Department of Mathematics and Applied Mathematics, University of Pretoria, Private Bag X20, Hatfield 0028, South Africa}
	\email{cmschwanke26@gmail.com}
	\date{\today}
	
	\begin{abstract}
		Given any finite subset $A$ of order $n$ of a distributive lattice and $k\in\{1,...,n\}$, there is a natural extension of the median operation to $n$ variables which generalizes the notion of the $k$th smallest element of $A$. By applying each of these operations to $A$, a totally ordered set $to(A)$ is obtained. We refer to $to(A)$ as the total orderization of $A$. After developing a brief theory of total orderization invariant maps on distributive lattices, it is shown in this paper how these functions generalize and provide new characterizations for symmetric continuous positively homogeneous functions, bounded orthosymmetric multilinear maps, and certain power sum polynomials on vector lattices. These theorems generalize several results by Bernau, Huijsmans, Kusraev, Azouzi, Boulabiar, Buskes, Boyd, Ryan, and Snigireva and in turn reveal novel properties of the various maps studied in this paper.
	\end{abstract}
	
	\maketitle
	
	\section{Introduction}\label{S:intro}
	
	Given a nonempty set $X$, the space $\mathbb{R}^X$ of real-valued functions on $X$ is a distributive lattice under the pointwise ordering. For $n\in\mathbb{N}$ and $f_1,...,f_n\in\mathbb{R}^X$, we know that $\bigwedge_{k=1}^nf_k$ is given by the pointwise minimum, and $\bigvee_{k=1}^nf_k$ is specified by the pointwise maximum. More generally, for any $k\in\mathbb{N}$ with $k\leq n$, the function
	\[
	\mathcal{M}_k(f_1,...,f_n)(x)=\Bigl\{\textnormal{the $k$th smallest value of}\ \{f_1(x),...,f_n(x)\}\Bigr\}
	\]
	is also element of $\mathbb{R}^X$. We can naturally refer to the totally ordered subset of $\mathbb{R}^X$
	\[
	\{\mathcal{M}_1(f_1,...,f_n),\mathcal{M}_2(f_1,...,f_n),...,\mathcal{M}_n(f_1,...,f_n)\}
	\]
	as the \textit{total orderization} of $\{f_1,...,f_n\}$.
	
	In \cite[Section~2]{BoyRySniga}, Boyd, Ryan, and Snigireva provide natural extensions of the median operation on a distributive lattice to expand the notion of these $\mathcal{M}_k$ functions, and hence the concept of total orderizations, to Banach lattices. Though they use a different notation than this paper, they essentially define for $k\leq n$, a Banach lattice $E$, and $f_1,...,f_n\in E$
	\[
	\mathcal{M}_k(f_1,...,f_n)=\underset{\substack{1\leq i_1,i_2,...,i_{n+1-k}\leq n \\ i_1<i_2<\cdots<i_{n+1-k}}}{\bigvee}\left(\bigwedge_{j=1}^{n+1-k}f_{i_j}\right).
	\]
	
	For $n=2$ however, the situation is more simple, and the median operation need not be considered. Given a distributive lattice $L$ and $x,y\in L$, the total orderization of $\{x,y\}$ is simply $\{x\wedge y,x\vee y\}$. Perusing through the literature on vector lattices, one can observe that these total orderizations of order two possess some intriguing invariance properties. 
	
	The most elementary of these occurrences is that the sum of two elements remains invariant under total orderizations. Specifically, for a vector lattice $E$ and $f,g\in E$, we know that $f+g=(f\wedge g)+(f\vee g)$.
	
	Interestingly, Bernau and Huijsmans showed in \cite[Proposition~1.13]{BerHui2} that if $A$ is an Archimedean almost $f$-algebra, then the multiplication on $A$ is invariant under total orderizations, that is
	\[
	ab=(a\wedge b)(a\vee b)\quad (a,b\in A).
	\]
	More generally, Kusraev proved in \cite[Proposition~1]{Kus3} that a positive bilinear map $T\colon E\times E\to F$, with $E$ and $F$ Archimedean vector lattices, is orthosymmetric if and only if
	\[
	T(x,y)=T(x\wedge y, x\vee y)\quad (x,y\in E).
	\]
	
	Later, Azouzi, Boulabiar, and Buskes showed that the geometric mean, a symmetric continuous positively homogeneous function, is also invariant under total orderizations. Indeed, if $E$ is an Archimedean geometric mean closed vector lattice and
	\[
	f\boxtimes g:=\frac{1}{2}\inf\{\theta f+\theta^{-1}g\ :\ \theta\in(0,\infty)\}\quad (f,g\in E^+),
	\]
	then it is proven in \cite[Lemma 4.1]{AzBoBus} that
	\[
	f\boxtimes g=(f\wedge g)\boxtimes(f\vee g)\quad (f,g\in E^+).
	\]
	
	Extending beyond $n=2$, Boyd, Ryan, and Snigireva recently extended Kusraev's result above to several variables in \cite[Proposition~2]{BoyRySniga} for regular $n$-linear forms on a Banach lattice. Specifically, they show that if $E$ is a Banach lattice and $A\colon E^n\to\mathbb{R}$ is a regular $n$-linear form, then $A$ is orthosymmetric if and only if
	\[
	A(f_1,...,f_n)=A\Bigl(\mathcal{M}_1(f_1,...,f_n),\mathcal{M}_2(f_1,...,f_n),...,\mathcal{M}_n(f_1,...,f_n)\Bigr)
	\]
	for every $f_1,...,f_n\in E$. Using the terminology of this paper, Boyd, Ryan, and Snigireva prove here that $A$ is orthosymmetric if and only if $A$ is \textit{total orderization invariant}.
	
	In this paper we further investigate the common thread of invariance under total orderizations in these results. We begin by introducing a brief theory for maps on distributive lattices that are total orderization invariant in Section~\ref{S:maps}. Specifically, we show in Proposition~\ref{P:invarimpsym} that every map defined on a distributive lattice that is total orderization invariant is symmetric (though the converse does not hold) and provide in Theorems~\ref{T:genorthosym}\&\ref{T:genorthsteady} two characterizations for total orderization invariant maps defined on a distributive lattice with a smallest element.
	
	In Section~\ref{S:MRs} we then illustrate that total orderization invariant maps generalize the notions of symmetric continuous positively homogeneous functions (Theorem~\ref{T:funcal}), bounded orthosymmetric multilinear maps between a uniformly complete vector lattice and a separated convex bornological space (Theorem~\ref{T:ortho}), and the related and to-be-introduced bounded orthogonally steady power sum polynomials on vector lattices (Theorem~\ref{T:vals}). In the same breath, characterizations for each of these types of maps in terms of total orderizations are provided.  We achieve these results by exploiting not only the theory developed in Section~\ref{S:maps} of this paper but also the theory of functional calculus given in \cite{BusdPvR,BusSch} as well as some results on bounded orthosymmetric maps and orthogonally additive polynomials found in \cite{Kusa2, Kusa}.
	
	The main theorems in this paper generalize the results by Bernau, Huijsmans, Kusraev, Azouzi, Boulabiar, Buskes, Boyd, Ryan, and Snigireva outlined previously in this section. The utility of the theory outlined in this paper can be further seen in Corollary~\ref{C:sum}, Corollary~\ref{C:orthofuncal}, Corollary~\ref{C:steadyfuncal}, Corollary~\ref{C:troitsky}, and Corollary~\ref{C:vals}, where novel properties of these maps under consideration are given.
	
	We assume the reader is familiar with the basics of lattice and vector lattice theory, but the reader is referred to \cite{AB, Birk, LuxZan1, Zan2} for any unexplained terminology or basic facts if necessary. In this paper $\mathbb{N}$ stands for the set of strictly positive integers, while $\mathbb{R}$ denotes the ordered field of real numbers.
	
	\section{Total orderization invariant maps}\label{S:maps}
	
	We begin this section with some convenient notation.
	
	\begin{notation}
		Let $L$ denote a nonempty distributive lattice throughout this section.
	\end{notation}
	
	\begin{notation}
		Given $n\in\mathbb{N}$, we write $[n]:=\{1,...,n\}$ for short in this paper.
	\end{notation}
	
	The median operation on $L$ (see e.g. \cite[Chapter~II, Section 6]{Birk}) is defined for $x_1,x_2,x_3\in L$ by
	\[
	(x_1,x_2,x_3)=(x_1\wedge x_2)\vee(x_1\wedge x_3)\vee(x_2\wedge x_3)=(x_1\vee x_2)\wedge(x_1\vee x_3)\wedge(x_2\vee x_3).
	\]
	Defining for $x_1,x_2,x_3\in L$, and $k\in[3]$
	\[
	\mathcal{M}_k(x_1,x_2,x_3)=\underset{\substack{i_1,i_2,...,i_{4-k}\in[3] \\ i_1<i_2<\cdots<i_{4-k}}}{\bigvee}\left(\bigwedge_{j=1}^{4-k}x_{i_j}\right)=\underset{\substack{i_1,i_2,...,i_{4-k}\in[3] \\ i_1<i_2<\cdots<i_{4-k}}}{\bigwedge}\left(\bigvee_{j=1}^{4-k}x_{i_j}\right)
	\]
	and simplifying these expressions, we get
	\[
	\mathcal{M}_1(x_1,x_2,x_3)=x_1\wedge x_2\wedge x_3,\ \mathcal{M}_2(x_1,x_2,x_3)=(x_1,x_2,x_3),\ \textnormal{and}\
	\mathcal{M}_3(x_1,x_2,x_3)=x_1\vee x_2\vee x_3.
	\]
	
	\begin{remark}\label{R:three}
		It is readily checked that
		\begin{itemize}
			\item[$(i)$] $\mathcal{M}_k$ is symmetric for each $k\in[3]$,
			\item[$(ii)$] if $x_1,x_2,x_3\in L$ and $x_1\leq x_2\leq x_3$, then $\mathcal{M}_k=x_k$ for every $k\in[3]$, and
			\item[$(iii)$] $\mathcal{M}_1(x_1,x_2,x_3)\leq\mathcal{M}_2(x_1,x_2,x_3)\leq\mathcal{M}_3(x_1,x_2,x_3)$ for all $x_1,x_2,x_3\in L$.
		\end{itemize}
	\end{remark}
	
	These $\mathcal{M}_k$ operations naturally extend to any finite number of variables. As mentioned in the introduction, they appear in \cite[Section~2]{BoyRySniga} under a slightly different notation.
	
	\begin{definition}\label{D:ot}
		For $n\in\mathbb{N}$, $x_1,...,x_n\in L$, and $k\in[n]$ let
		\[
		\mathcal{M}_k(x_1,...,x_n)=\underset{\substack{i_1,i_2,...,i_{n+1-k}\in[n] \\ i_1<i_2<\cdots<i_{n+1-k}}}{\bigvee}\left(\bigwedge_{j=1}^{n+1-k}x_{i_j}\right)=\underset{\substack{i_1,i_2,...,i_{n+1-k}\in[n] \\ i_1<i_2<\cdots<i_{n+1-k}}}{\bigwedge}\left(\bigvee_{j=1}^{n+1-k}x_{i_j}\right).
		\]
	\end{definition}
	
	In Proposition~\ref{P:3ton} we show that the contents of Remark~\ref{R:three} extend to any finite number of variables. The author suspects this result is possibly known but was unable to locate a reference.
	
	\begin{proposition}\label{P:3ton}
		Let $n\in\mathbb{N}$. Then
		\begin{itemize}
			\item[$(i)$] $\mathcal{M}_k\colon L^n\to L$ is a symmetric function for all $k\in[n]$,
			\item[$(ii)$] if $x_1,x_2,...,x_n\in L$ satisfy $x_1\leq x_2\leq\cdots\leq x_n$, then $\mathcal{M}_k(x_1,...,x_n)=x_k$ for all $k\in[n]$, and
			\item[$(iii)$] $\mathcal{M}_1(x_1,...,x_n)\leq\mathcal{M}_2(x_1,...,x_n)\leq\cdots\leq\mathcal{M}_n(x_1,...,x_n)$ for all $x_1,x_2,...,x_n\in L$.
		\end{itemize}
	\end{proposition}
	
	\begin{proof}
		$(i)$ Fix $k\in[n]$, and let $x_1,x_2,...,x_n\in L$ be arbitrary. It suffices to show that for any $i,j\in[n]$ with $i<j$ we have
		\[
		\mathcal{M}_k(x_1,...,x_i,...,x_j,...,x_n)=\mathcal{M}_k(x_1,...,x_j,...,x_i,...,x_n).
		\]
		In order to minimize cumbersome notation, we will show that
		\[
		\mathcal{M}_k(x_1,x_2,...,x_n)=\mathcal{M}_k(x_2,x_1,...,x_n),
		\]
		noting that the more general proof is similar. For $n=1$ the result is trivial, and for $n=2$ the result is known since $\mathcal{M}_1(x_1,x_2)=x_1\wedge x_2$ and $\mathcal{M}_2(x_1,x_2)=x_1\vee x_2$. Assuming that $n\geq 3$, observe that
		\[
		\mathcal{M}_k(x_1,x_2,...,x_n)=R(x_1,x_2,...,x_n)\vee S(x_1,x_2,...,x_n)\vee T(x_1,x_2,...,x_n),
		\]
		where
		\[
		R(x_1,x_2,...,x_n)=\underset{\substack{3\leq q_1,q_2,...,q_{n-1-k}\leq n \\ q_1<q_2<\cdots<q_{n-1-k}}}{\bigvee}(x_1\wedge x_2\wedge x_{q_1}\wedge\cdots\wedge x_{q_{n-1-k}}),
		\]
		and
		\[
		S(x_1,x_2,...,x_n)=
		\]
		\[
		\left(\underset{\substack{3\leq r_1,r_2,...,r_{n-k}\leq n \\ r_1<r_2<\cdots<r_{n-k}}}{\bigvee}(x_1\wedge x_{r_1}\wedge\cdots\wedge x_{r_{n-k}})\right)\bigvee\left(\underset{\substack{3\leq s_1,s_2,...,s_{n-k}\leq n \\ s_1<s_2<\cdots<s_{n-k}}}{\bigvee}(x_2\wedge x_{s_1}\wedge\cdots\wedge x_{s_{n-k}})\right),
		\]
		and
		\[
		T(x_1,x_2,...,x_n)=\underset{\substack{4\leq t_1,t_2,...,t_{n-k}\leq n \\ t_1<t_2<\cdots<t_{n-k}}}{\bigvee}(x_3\wedge x_{t_1}\wedge\cdots\wedge x_{t_{n-k}}).
		\]
		Clearly, we see that $R(x_1,x_2,...,x_n)=R(x_2,x_1,...,x_n),\quad  S(x_1,x_2,...,x_n)=S(x_2,x_1,...,x_n)$, and $T(x_1,x_2,...,x_n)=T(x_2,x_1,...,x_n)$. It follows that
		\[
		\mathcal{M}_k(x_1,x_2,...,x_n)=\mathcal{M}_k(x_2,x_1,...,x_n).
		\]
		We conclude that $\mathcal{M}_k$ is symmetric.
		
		$(ii)$ Suppose $x_1,x_2,...,x_n\in L$ satisfy $x_1\leq x_2\leq\cdots\leq x_n$, and let $k\in[n]$ be arbitrary. Then we obtain
		\[
		\mathcal{M}_k(x_1,...,x_n)=\underset{\substack{i_1,i_2,...,i_{n+1-k}\in[n] \\ i_1<i_2<\cdots<i_{n+1-k}}}{\bigvee}\left(\bigwedge_{j=1}^{n+1-k}x_{i_j}\right)=\bigwedge_{i=k}^{n}x_{i}=x_k.
		\]
		
		$(iii)$ For $n=1$, the result is trivial. Suppose $n\geq 2$ and fix $k\in[n-1]$. We prove that $\mathcal{M}_k(x_1,...,x_n)\leq\mathcal{M}_{k+1}(x_1,...,x_n)$, that is, that
		\begin{align*}
			\underset{\substack{i_1,i_2,...,i_{n+1-k}\in[n] \\ i_1<i_2<\cdots<i_{n+1-k}}}{\bigvee}\left(\bigwedge_{j=1}^{n+1-k}x_{i_j}\right)\leq\underset{\substack{t_1,t_2,...,t_{n-k}\in[n] \\ t_1<t_2<\cdots<t_{n-k}}}{\bigvee}\left(\bigwedge_{l=1}^{n-k}x_{t_l}\right).
		\end{align*}
		To this end, let $i_1,i_2,...,i_{n+1-k}\in[n]$ with $i_1<i_2<\cdots<i_{n+1-k}$ arbitrary. Clearly, there exist $t_1,t_2,...,t_{n-k}\in[n]$ with $t_1<t_2<\cdots<t_{n-k}$ for which 
		\[
		\{t_1,t_2,...,t_{n-k}\}\subseteq\{i_1,i_2,...,i_{n+1-k}\}.
		\]
		By the monotonicity of infima, we obtain
		\[
		\bigwedge_{j=1}^{n+1-k}x_{i_j}\leq\bigwedge_{l=1}^{n-k}x_{t_l}\leq\underset{\substack{t_1,t_2,...,t_{n-k}\in[n] \\ t_1<t_2<\cdots<t_{n-k}}}{\bigvee}\left(\bigwedge_{l=1}^{n-k}x_{t_l}\right).
		\]
		Since the inequality
		\[
		\bigwedge_{j=1}^{n+1-k}x_{i_j}\leq\underset{\substack{t_1,t_2,...,t_{n-k}\in[n] \\ t_1<t_2<\cdots<t_{n-k}}}{\bigvee}\left(\bigwedge_{l=1}^{n-k}x_{t_l}\right)
		\]
		holds for all $i_1,i_2,...,i_{n+1-k}\in[n]$ with $i_1<i_2<\cdots<i_{n+1-k}$, we conclude that
		\begin{align*}
			\underset{\substack{i_1,i_2,...,i_{n+1-k}\in[n] \\ i_1<i_2<\cdots<i_{n+1-k}}}{\bigvee}\left(\bigwedge_{j=1}^{n+1-k}x_{i_j}\right)\leq\underset{\substack{t_1,t_2,...,t_{n-k}\in[n] \\ t_1<t_2<\cdots<t_{n-k}}}{\bigvee}\left(\bigwedge_{l=1}^{n-k}x_{t_l}\right).
		\end{align*}
		It follows that
		\[
		\mathcal{M}_1(x_1,...,x_n)\leq\mathcal{M}_2(x_1,...,x_n)\leq\cdots\leq\mathcal{M}_n(x_1,...,x_n).
		\]
	\end{proof}
	
	The following terminology is motivated by Proposition~\ref{P:3ton}$(iii)$ above.
	
	\begin{definition}
		Given $n\in\mathbb{N}$ and $x_1,...,x_n\in L$, we write
		\[
		to(\{x_1,...,x_n\}):=\{\mathcal{M}_1(x_1,...,x_n),\mathcal{M}_2(x_1,...,x_n),\dots,\mathcal{M}_n(x_1,...,x_n)\}
		\]
		and call the totally ordered set $to(\{x_1,...,x_n\})$ the \textit{total orderization} of $\{x_1,...,x_n\}$.
	\end{definition}
	
	\begin{remark}\label{R:totordsetssame}
		It is an immediate consequence of Proposition~\ref{P:3ton} that if $\{x_1,...,x_n\}\subseteq L$, then $\{x_1,...,x_n\}$ is totally ordered if and only if
		\[
		to(\{x_1,...,x_n\})=\{x_1,...,x_n\}.
		\]
	\end{remark}
	
	We next provide a formal definition of maps on distributive lattices that are invariant under total orderizations. These functions constitute our primary focus throughout the rest of the paper.
	
	\begin{definition}
		Let $n\in\mathbb{N}$, and suppose $A$ is a nonempty set. A map $T\colon L^n\to A$ is said to be \textit{total orderization invariant} if
		\[
		T(x_1,...,x_n)=T\Bigl(\mathcal{M}_1(x_1,...,x_n),\mathcal{M}_2(x_1,...,x_n),\dots,\mathcal{M}_n(x_1,...,x_n)\Bigr)
		\]
		holds for every $x_1,...,x_n\in L$.
	\end{definition}
	
	Our first result regarding these maps is useful despite its simplicity.
	
	\begin{proposition}\label{P:invarimpsym}
		Fix $n\in\mathbb{N}$. Let $A$ be a nonempty set. If a map $T\colon L^n\to A$ is total orderization invariant, then $T$ is symmetric.
	\end{proposition}
	
	\begin{proof}
		Suppose $T\colon L^n\to A$ is total orderization invariant. Let $x_1,...,x_n\in L$. By assumption we have
		\[
		T(x_1,...,x_n)=T\Bigl(\mathcal{M}_1(x_1,...,x_n),\mathcal{M}_2(x_1,...,x_n),\dots,\mathcal{M}_n(x_1,...,x_n)\Bigr).
		\]
		Also, $\mathcal{M}_k$ is symmetric for each $k\in[n]$ by Proposition~\ref{P:3ton}$(i)$. Thus for any permutation $\sigma$ on $[n]$ we obtain
		\begin{align*}
			T&(x_1,...,x_n)=T\Bigl(\mathcal{M}_1(x_1,...,x_n),\mathcal{M}_2(x_1,...,x_n),\dots,\mathcal{M}_n(x_1,...,x_n)\Bigr)\\
			&=T\Biggl(\mathcal{M}_1\Bigl(x_{\sigma(1)},...,x_{\sigma(n)}\Bigr),\mathcal{M}_2\Bigl(x_{\sigma(1)},...,x_{\sigma(n)}\Bigr),\dots,\mathcal{M}_n\Bigl(x_{\sigma(1)},...,x_{\sigma(n)}\Bigr)\Biggr)\\
			&=T\Bigl(x_{\sigma(1)},...,x_{\sigma(n)}\Bigr).
		\end{align*}
		This concludes the proof.
	\end{proof}
	
	The converse of Proposition~\ref{P:invarimpsym} does not hold, as the following counterexample illustrates.
	
	\begin{example}
		Consider the Banach lattice $C[0,1]$ under the supremum norm. Define $T\colon C[0,1]\to\mathbb{R}$ by $T(f,g)=\|f\|+\|g\|$. Clearly, $T$ is symmetric. However, for $f,g\in C[0,1]$ defined by $f(x)=1-x$ and $g(x)=x$, we have $T(f,g)=2$ and $T(f\wedge g,f\vee g)=\frac{3}{2}$. Thus $T$ is not total orderization invariant.
	\end{example}
	
	We next proceed to the main results of this section. Theorem~\ref{T:genorthosym} below will be employed in Theorem~\ref{T:ortho} to obtain a characterization of bounded orthosymmetric multilinear maps on vector lattices.
	
	\begin{theorem}\label{T:genorthosym}
		Fix $n\in\mathbb{N}$. Suppose $L$ possesses a smallest element $\theta$, and let $A$ be a nonempty set. Put $a\in A$, and assume $T\colon L^n\to A$ is total orderization invariant. The following are equivalent.
		\begin{itemize}
			\item[$(i)$] $T(x_1,...,x_n)=a$ whenever $x_1,...,x_n\in L$ and $x_i\wedge x_j=\theta$ for some $i,j\in[n]$, and
			\item[$(ii)$] $T(\theta,x_1,...,x_{n-1})=a$ holds for all $x_1,...,x_{n-1}\in L$.
		\end{itemize}
	\end{theorem}
	
	\begin{proof}
		The implication $(i)\implies(ii)$ is evident since $\theta$ is the smallest element of $L$. In order to prove $(ii)\implies(i)$, suppose that $T(\theta,x_1,...,x_{n-1})=a$ for every $x_1,...,x_{n-1}\in L$. To show that $(i)$ holds, let $x_1,...,x_n\in L$ be such that $x_i\wedge x_j=\theta$ for some $i,j\in[n]$. Then we have
		\[
		\mathcal{M}_1(x_1,...,x_n)=\bigwedge_{k=1}^nx_k=\theta.
		\]
		Hence from assumption $(ii)$ we obtain
		\begin{align*}
			T(x_1,...,x_n)&=T\Bigl(\mathcal{M}_1(x_1,...,x_n),\mathcal{M}_2(x_1,...,x_n),\dots,\mathcal{M}_n(x_1,...,x_n)\Bigr)\\
			&=T\Bigl(\theta,\mathcal{M}_2(x_1,...,x_n),\dots,\mathcal{M}_n(x_1,...,x_n)\Bigr)=a.
		\end{align*}
	\end{proof}
	
	We conclude this section with the following essential tool for the proof of Theorem~\ref{T:vals}.
	
	\begin{theorem}\label{T:genorthsteady}
		Fix $n\in\mathbb{N}$.  Suppose $L$ possesses a smallest element $\theta$, and let $A$ be a nonempty set. Let $\phi\colon L\to A$ be an arbitrary function, and assume $T\colon L^n\to A$ is total orderization invariant. The following are equivalent.
		\begin{itemize}
			\item[$(i)$] $T(x_1,...,x_n)=\phi\left(\bigvee_{k=1}^nx_k\right)$ whenever $x_1,...,x_n\in L$ satisfy $x_i\wedge x_j=\theta$ for all $i,j\in[n]$ with $i\neq j$, and
			\item[$(ii)$] $T(\underbrace{\theta,\dots,\theta}_{n-1\ \text{copies}},x)=\phi(x)$ holds for all $x\in L$.
		\end{itemize}
	\end{theorem}
	
	\begin{proof}
		The implication $(i)\implies(ii)$ is trivial. In order to verify $(ii)\implies(i)$, assume that $T(\underbrace{\theta,\dots,\theta}_{n-1\ \text{copies}},x)=\phi(x)$ holds for each $x\in L$. Let $x_1,...,x_n\in L$ satisfy $x_i\wedge x_j=\theta$ for all $i,j\in[n]$ with $i\neq j$. Then
		\[
		\mathcal{M}_{n-1}(x_1,...,x_n)=\underset{\substack{i_1,i_2\in[n]\\ i_1<i_2}}{\bigvee}\left(\bigwedge_{j=1}^{2}x_{i_j}\right)=\theta.
		\]
		It then follows from Proposition~\ref{P:3ton}$(iii)$ that $\mathcal{M}_{k}(x_1,...,x_n)=\theta$ for all $k\in[n-1]$. Thus we attain
		\begin{align*}
			T(x_1,...,x_n)&=T\Bigl(\mathcal{M}_1(x_1,...,x_n),\mathcal{M}_2(x_1,...,x_n),\dots,\mathcal{M}_n(x_1,...,x_n)\Bigr)\\
			&=T\left(\underbrace{\theta,\dots,\theta}_{n-1\ \text{copies}},\bigvee_{k=1}^nx_k\right)\\
			&=\phi\left(\bigvee_{k=1}^nx_k\right).
		\end{align*}
	\end{proof}
	
	\section{Total orderization invariant maps on vector lattices}\label{S:MRs}
	
	In this final section of the paper, we consider Archimedean real vector lattices instead of the more general distributive lattices. The primary aim of this section is to illustrate that symmetric continuous positively homogeneous functions, bounded orthosymmetric multilinear maps, and what we call bounded orthogonally steady power sum polynomials are all total orderization invariant. In fact, we prove that invariance under total orderizations actually characterizes these maps and reveal some new properties of these functions along the way.
	
	\begin{notation}
		Throughout the remainder of this paper, $E$ denotes an Archimedean real vector lattice that is nontrivial; i.e. $E\neq\{0\}$.
	\end{notation}
	
	We extensively utilize the Archimedean vector lattice functional calculus introduced by Buskes, de Pagter, and van Rooij in \cite[Section~3]{BusdPvR}. Using the notation in \cite{BusdPvR}, we denote by $\mathcal{H}(\mathbb{R}^{n})$ the space of all continuous, real-valued functions $h$ on $\mathbb{R}^{n}$ that are \textit{positively homogeneous}, i.e. $h(\lambda x)=\lambda h(x)$ for every $\lambda
	\in \mathbb{R}^{+}$ and all $x\in \mathbb{R}^{n}$.
	
	\begin{definition}\cite[Definition~3.1]{BusdPvR}
		Given $f_{1},...,f_{n},g\in E$ and $h\in\mathcal{H}(\mathbb{R}^{n})$, we write $h(f_{1},...,f_{n})=g$ when
		\[
		h(\omega
		(f_{1}),...,\omega (f_{n}))=\omega(g)
		\]
		holds for every nonzero real-valued vector lattice homomorphism $\omega$ defined on the vector sublattice of $E$ generated by $\{f_{1},...,f_{n},g\}$.
	\end{definition}
	
	This functional calculus cannot always be defined in certain vector lattices however. For example, if $h(x,y)=\sqrt{x^2+y^2}$, then there exists $f$ and $g$ in the Fremlin tensor product $C[0,1]\bar{\otimes}C[0,1]$ such that $h(f,g)$ is undefined (see \cite[Theorem 4.10]{BusSch2}). We thus turn to the notion of $h$-completeness, as introduced in \cite{BusSch}.
	
	\begin{definition}\cite[Definition~3.2]{BusSch}
		For $h\in\mathcal{H}(\mathbb{R}^{n})$,
		we say that $E$ is $h$-\textit{complete} if for every $f_{1},...,f_{n}\in E$ there exists $g\in E$ such that $h(f_{1},...,f_{n})=g$. 
	\end{definition}
	
	\begin{remark}
		By \cite[Theorem~3.7]{BusdPvR}, if $E$ is uniformly complete, then $E$ is $h$-complete for every $h\in\bigcup_{n\in\mathbb{N}}\mathcal{H}(\mathbb{R}^n)$.
	\end{remark}
	
	We next show that symmetric continuous positively homogeneous functions $h\colon E^n\to E$ defined via the Archimedean functional calculus are total orderization invariant. As a matter of fact we prove that invariance under total orderizations characterizes the symmetry of these maps. This result is required for our proof of Theorem~\ref{T:vals}.
	
	\begin{theorem}\label{T:funcal}
		Fix $n\in\mathbb{N}$. Let $h\in\mathcal{H}(\mathbb{R}^n)$, and assume that $E$ is $h$-complete. The following are equivalent.
		\begin{itemize}
			\item[$(i)$] $h\colon\mathbb{R}^n\to\mathbb{R}$ is symmetric.
			\item[$(ii)$] The function $h\colon E^n\to E$ defined via the Archimedean functional calculus is symmetric.
			\item[$(iii)$] The function $h\colon E^n\to E$ defined via the Archimedean functional calculus is total orderization invariant.
		\end{itemize}
	\end{theorem}
	
	\begin{proof}	
We first prove that $(ii)\implies (i)$. To this end, suppose that $h$ is symmetric on $E$, and consider a permutation $\sigma$ on $[n]$. We have that
		\begin{align*}
			h(f_1,...,f_n)=h\Bigl(f_{\sigma(1)},...,f_{\sigma(n)}\Bigr)
		\end{align*}
		holds for all $f_1,...,f_n\in E$. Fix $f\in E\setminus\{0\}$, and let $A$ be the vector sublattice of $E$ generated by $\{f\}$. Also let $\omega$ be a nonzero real-valued vector lattice homomorphism on $A$. Adjusting by a scalar multiple if necessary, we can assume that $\omega(f)=1$. Next let $x_1,...,x_n\in\mathbb{R}$. Invoking \cite[Theorem~3.11]{BusSch} and our assumption $(ii)$, we get
		
		\begin{align*}
			h(x_1,...,x_n)&=h\Bigl(\omega(x_1f),...,\omega(x_nf)\Bigr)=\omega\Bigl(h(x_1f,...,x_nf)\Bigr)=\omega\Bigl(h(x_{\sigma(1)}f,...,x_{\sigma(n)}f)\Bigr)\\
			&=h\Bigl(\omega(x_{\sigma(1)}f),...,\omega(x_{\sigma(n)}f)\Bigr)=h(x_{\sigma(1)},...,x_{\sigma(n)}).
		\end{align*}
		Hence $h\colon\mathbb{R}^n\to\mathbb{R}$ is symmetric.
		
		We next show that $(i)\implies(iii)$. To this end, assume that $h\colon\mathbb{R}^n\to\mathbb{R}$ is symmetric. Let $f_1,...,f_n\in E$, and let $A$ be the vector sublattice of $E$ generated by
		\[
		\left\{f_1,...,f_n,h(f_1,...,f_n),h\Bigl(\mathcal{M}_1(f_1,...,f_n),\mathcal{M}_2(f_1,...,f_n),\dots,\mathcal{M}_n(f_1,...,f_n)\Bigr)\right\}.
		\]
		Suppose that $\omega\colon A\to\mathbb{R}$ is a non-zero vector lattice homomorphism. Using \cite[Theorem~3.11]{BusSch}, the fact that $\{\omega(f_1),...,\omega(f_n)\}$ is a totally ordered set of real numbers along with Remark~\ref{R:totordsetssame}, and the assumption that $h\colon\mathbb{R}^n\to\mathbb{R}$ is symmetric, we have
		\begin{align*}
			\omega&\Bigl(h(f_1,...,f_n)\Bigr)=h\Bigl(\omega(f_1),...,\omega(f_n)\Bigr)\\
			&=h\Biggl(\mathcal{M}_1\Bigl(\omega(f_1),...,\omega(f_n)\Bigr),\mathcal{M}_2\Bigl(\omega(f_1),...,\omega(f_n)\Bigr),\dots,\mathcal{M}_n\Bigl(\omega(f_1),...,\omega(f_n)\Bigr)\Biggr)\\
			&=h\Biggl(\omega\Bigl(\mathcal{M}_1(f_1,...,f_n)\Bigr),\omega\Bigl(\mathcal{M}_2(f_1,...,f_n)\Bigr),\dots,\omega\Bigl(\mathcal{M}_n(f_1,...,f_n)\Bigr)\Biggr)\\
			&=\omega\Biggl(h\Bigl(\mathcal{M}_1(f_1,...,f_n),\mathcal{M}_2(f_1,...,f_n),\dots,\mathcal{M}_n(f_1,...,f_n)\Bigr)\Biggr).
		\end{align*}
		Since the set of all nonzero vector lattice homomorphisms $\omega\colon A\to\mathbb{R}$ separate the points of $A$, we obtain
		\[
		h(f_1,...,f_n)=h\Bigl(\mathcal{M}_1(f_1,...,f_n),\mathcal{M}_2(f_1,...,f_n),\dots,\mathcal{M}_n(f_1,...,f_n)\Bigr).
		\]
		
		Finally note that the implication $(iii)\implies(ii)$ follows from Proposition~\ref{P:invarimpsym}. This concludes the proof.
	\end{proof}
	
	Since the summation function $h(x_1,...,x_n)=\sum_{k=1}^nx_k\quad (x_1,...,x_n\in\mathbb{R})$ is symmetric, continuous, and positively homogeneous, we acquire the following corollary which extends the elementary result that
	\[
	f+g=(f\wedge g)+(f\vee g)
	\]
	holds for all $f,g\in E$.
	
	\begin{corollary}\label{C:sum}
		Let $n\in\mathbb{N}$. Then
		\[
		\sum_{k=1}^{n}f_k=\sum_{k=1}^{n}\mathcal{M}_k(f_1,...,f_n)
		\]
		holds for all $f_1,..,f_n\in E$.
	\end{corollary}
	
	In order to showcase the utility of the theorems in this paper so far, we next provide two corollaries which follow immediately from Theorem~\ref{T:funcal} as well as Theorem~\ref{T:genorthosym} and Theorem~\ref{T:genorthsteady}, respectively. 	We note that Corollary~\ref{C:orthofuncal} contains the geometric means $\displaystyle\mathfrak{G}(x_1,...,x_n)=\sqrt[n]{\prod_{k=1}^{n}|x_k|}$ and the infimum $\mathcal{M}_1$, while Corollary~\ref{C:steadyfuncal} encompasses functions of the form $\displaystyle\mathfrak{S}(x_1,...,x_n)=\sqrt[n]{\sum_{k=1}^nx_k^n}$ (which appear later in this paper), the supremum $\mathcal{M}_n$, and the summation function.
	
	\begin{corollary}\label{C:orthofuncal}
		Fix $n\in\mathbb{N}$. Let $h\in\mathcal{H}(\mathbb{R}^n)$ be symmetric, and assume that $E$ is $h$-complete. The following are equivalent.
		\begin{itemize}
			\item[$(i)$] $h(f_1,...,f_n)=0$ whenever $f_1,...,f_n\in E^+$ and $f_i\wedge f_j=0$ for some $i,j\in[n]$, and
			\item[$(ii)$] $h(0,f_1,...,f_{n-1})=0$ holds for all $f_1,...,f_{n-1}\in E^+$.
		\end{itemize}
	\end{corollary}
	
	\begin{corollary}\label{C:steadyfuncal}
		Fix $n\in\mathbb{N}$. Let $h\in\mathcal{H}(\mathbb{R}^n)$ be symmetric, and assume that $E$ is $h$-complete. The following are equivalent.
		\begin{itemize}
			\item[$(i)$] $h(f_1,...,f_n)=\bigvee_{k=1}^nf_k$ whenever $\{f_1,...,f_n\}$ is a pairwise disjoint subset of $E^+$, and
			\item[$(ii)$] $h(\underbrace{0,\dots,0}_{n-1\ \text{copies}},f)=f$ holds for all $f\in E^+$.
		\end{itemize}
	\end{corollary}
	
Next we shift our focus to orthosymmetric multilinear maps, orthogonally additive polynomials, and orthogonally steady power sum polynomials on vector lattices.
	
	\begin{definition}
		Given $n\in\mathbb{N}$ and a vector space $V$, we call an $n$-linear map $T\colon E^n\to V$ \textit{orthosymmetric} if $T(f_1,...,f_n)=0$ holds whenever $f_1,...,f_n\in E$ and $|f_i|\wedge|f_j|=0$ for some $i,j\in[n]$. An $n$-homogeneous polynomial $P\colon E\to V$ is said to be \textit{orthogonally additive} if $P(f+g)=P(f)+P(g)$ holds for all $f,g\in E$ with $|f|\wedge|g|=0$. If $P(f+g)=P(f)+P(g)$ holds for all $f,g\in E^+$ with $f\wedge g=0$, then we say that $P$ is \textit{positively orthogonally additive}. Likewise, we call a map $T\colon E^n\to V$ \textit{positively total orderization invariant} if the restriction of $T$ to $(E^+)^n$ is total orderization invariant.
	\end{definition}
	
	Fix $n,r\in\mathbb{N}$. Classically, a power sum symmetric polynomial of degree $n$ in $r$ variables is a map $S\colon\mathbb{R}^r\to\mathbb{R}$ of the form
	\[
	S(x_1,...,x_r)=\sum_{k=1}^{r}x_k^n.
	\]
	We naturally extend this notion to more general settings as follows.
	
	\begin{definition}
		Let $n,r\in\mathbb{N}$. For real vector spaces $U$ and $V$ and an $n$-homogeneous polynomial $P\colon U\to V$, define a symmetric map $S\colon U^r\to V$ by
		\[
		S(u_1,...,u_r)=\sum_{k=1}^{r}P(u_k).
		\]
		We call maps of this form \textit{power sum symmetric polynomials} of degree $n$ in $r$ variables or simply \textit{power sum polynomials} for short. The polynomial $P$ will be refereed to as the \textit{generating polynomial} for $S$. We call a power sum polynomial $S\colon E^r\to V$ \textit{orthogonally steady} if
		\[
		S(f_1,...,f_r)=S(0,f_1,...,f_i+f_j,...,f_r)
		\]
		holds whenever $f_1,...,f_r\in E$ and $i,j\in[r]$ with $i\neq j$ satisfy $|f_i|\wedge|f_j|=0$.
	\end{definition}
	
	The intimate relationship between orthogonally steady power sum polynomials of degree $n$ and orthogonally additive $n$-homogeneous polynomials is evident from the definitions. We formally state this fact below.
	
	\begin{lemma}\label{L:steadyiffadditive}
		Fix $n,r\in\mathbb{N}$ with $r\geq 2$. Let $V$ be a real vector space. A power sum polynomial $S\colon E^r\to V$ of degree $n$ in $r$ variables is orthogonally steady if and only if its generating $n$-homogeneous polynomial is orthogonally additive. 
	\end{lemma}
	
	In this section we reference the vector lattice $n$-power (for $n\in\mathbb{N}$) of an Archimedean vector lattice $E$, see \cite{BoBus}.
	
	\begin{notation}
		For $n\in\mathbb{N}$, the vector lattice $n$-power of $E$ will be denoted by \textnormal{$(E^{\textnormal{\textcircled{$n$}} },\textnormal{\textcircled{$n$}})$}.
	\end{notation}
	
	We in particular study bounded orthosymmetric $n$-linear maps $T\colon E^n\to Y$ and bounded orthogonally steady power sum polynomials of degree $n$ in $r$ variables $S\colon E^r\to Y$, where $E$ is a uniformly complete Archimedean vector lattice and $Y$ is a real separated convex bornological space. This setup is of profound interest to us because the theory of total orderization-invariance alone does not appear to be sufficient for achieving all of our desired results. We will thus receive aid from the following three results by Kusraeva as well some techniques used in the proofs of these theorems. For more information on separated convex bornological spaces in the present context, see \cite{Kusa2}.
	
	\begin{theorem}\cite[Theorem~4]{Kusa2}\label{T:KusaThm4}
		Let $E$ be uniformly complete, and let $Y$ be a real separated convex bornological space. If $n\in\mathbb{N}$ and $P\colon E\to Y$ is a bounded orthogonally additive $n$-homogeneous polynomial, then there exists a unique bounded linear operator $S\colon E^{\textnormal{\textcircled{$n$}}}\to Y$ such that
		\[
		P(f)=S\circ\textnormal{\textcircled{$n$}}(\underbrace{f,\dots,f}_{n\ \text{copies}})\quad (f\in E).
		\]
	\end{theorem}
	
	Actually, the proof of \cite[Theorem~4]{Kusa2} shows the following slightly more general result.
	
	\begin{theorem}\cite{Kusa2}\label{T:KusaProofofThm4}
		Let $E$ be uniformly complete, and let $Y$ be a real separated convex bornological space. If $n\in\mathbb{N}$ and $T\colon E^n\to Y$ is a bounded orthosymmetric $n$-linear map, then there exists a unique bounded linear operator $S\colon E^{\textnormal{\textcircled{$n$}}}\to Y$ such that
		\[
		T(f_1,...,f_n)=S\circ\textnormal{\textcircled{$n$}}(f_1,\dots,f_n)\quad (f_1,...,f_n\in E).
		\]
	\end{theorem}
	
	The following portion of the main result from \cite{Kusa}, also by Kusraeva, is only stated for positive $f_{1},\dots,f_{r}\in E$, yet Kusraeva's proof actually holds for all $f_{1},\dots,f_{r}\in E$. For the reader's convenience, we give an outline of the proof with some slight modifications.
	
	\begin{theorem}\cite{Kusa}\label{T:Kusa}
		Fix $n,r\in\mathbb{N}$. Let $E$ be uniformly complete, suppose $Y$ is a real separated convex bornological space, and let $P\colon E\to Y$ be a bounded orthogonally additive $n$-homogeneous polynomial. Define
		\[
		\mathfrak{S}(x_1,\dots,x_r)=\sqrt[n]{\sum_{k=1}^{r}x_k^n}\quad (x_1,\dots,x_r\in\mathbb{R}).
		\]
		Then
		\begin{equation}\label{eq:RMP}
			P\Bigl(\mathfrak{S}(f_{1},\dots,f_{r})\Bigr)=\sum_{k=1}^rP(f_{k})
		\end{equation}
		holds for all $f_{1},\dots,f_{r}\in E$, where the expression $\mathfrak{S}(f_{1},\dots,f_{r})$ in \eqref{eq:RMP} is defined via the Archimedean vector lattice functional calculus.
	\end{theorem}
	
	\begin{proof}
		Suppose that $P$ is orthogonally additive. Let $E^u$ denote the universal completion of $E$, which is well-known to be a semiprime $f$-algebra. Fix $f_1,\dots f_r\in E$. Let $C$ be the Archimedean $f$-subalgebra of $E^u$ generated by
		\[
		\left\{f_1,\dots,f_r,\mathfrak{S}(f_1,\dots,f_r)\right\}.
		\]
		Let $\omega\colon C\rightarrow\mathbb{R}$ be a nonzero multiplicative vector lattice homomorphism. Using \cite[Theorem~3.11]{BusSch} in the second equality below, we obtain
		\begin{align*}
			\omega\Biggl(\Bigl(\mathfrak{S}(f_1,...,f_r)\Bigr)^n\Biggr)&=\Biggl(\omega\Bigl(\mathfrak{S}(f_1,...,f_r)\Bigr)\Biggr)^n=\Biggl(\mathfrak{S}\Bigl(\omega(f_1),...,\omega(f_r)\Bigr)\Biggr)^n\\
			&=\left(\sqrt[n]{\sum_{k=1}^{r}\Bigl(\omega(f_k)\Bigr)^n}\right)^n=\sum_{k=1}^r\Bigl(\omega(f_k)\Bigr)^n=\omega\left(\sum_{k=1}^rf_k^n\right).
		\end{align*}
		Since the set of all nonzero multiplicative vector lattice homomorphisms $\omega\colon C\rightarrow\mathbb{R}$ separates the points of $C$ (see \cite[Corollary 2.7]{BusdPvR}), we have
		\begin{equation}\label{eq:trueinf-alg}
			\Bigl(\mathfrak{S}(f_1,...,f_r)\Bigr)^n=\sum_{k=1}^rf_k^n.
		\end{equation}
		
		By \cite[Theorem~4.1]{BoBus}, there exists a uniformly complete vector sublattice $F$ of $E^u$ and a vector lattice isomorphism $i\colon E^{\textnormal{\textcircled{$n$}} }\to F$ such that both
		\[
		\prod_{k=1}^{n}x_k\in F
		\]
		and
		\[
		i\circ\textnormal{\textcircled{$n$}} (x_1,\dots,x_n)=\prod_{k=1}^{n}x_k
		\]
		hold for all $x_1,\dots,x_n\in E$. It then follows from \eqref{eq:trueinf-alg} and the identity
		\[
		\textnormal{\textcircled{$n$}} (x_1,\dots,x_n)=i^{-1}\left(\prod_{k=1}^{n}x_k\right)\quad (x_1,...,x_n\in E)
		\]
		that
		\[
		\textnormal{\textcircled{$n$}} \Bigl(\underbrace{\mathfrak{S}(f_1,...,f_r),...,\mathfrak{S}(f_1,...,f_r)}_{n\ \text{copies}}\Bigr)=\sum_{k=1}^r\textnormal{\textcircled{$n$}}(\underbrace{f_k,...,f_k}_{n\ \text{copies}}).
		\]
		The result then follows from Theorem~\ref{T:KusaThm4}, as there exists a bounded linear operator $S\colon E^{\textnormal{\textcircled{$n$}} }\to Y$ for which
		\[
		P(x)=S\circ\textnormal{\textcircled{$n$}} (\underbrace{x,\dots,x}_{n\ \text{copies}})
		\]
		holds for every $x\in E$.
	\end{proof}
	
	We additionally require the following lemma, which is analogous to \cite[Proposition~1.13]{BerHui2} by Bernau and Huijsmans, mentioned in the introduction. The proof consists of a familiar point-separating argument, similar to what was presented in the proof of Theorem~\ref{T:funcal}. A more general result is later proved in Theorem~\ref{T:ortho}.
	
	\begin{lemma}\label{L:mult}
		Let $n\in\mathbb{N}$, and let $A$ be an Archimedean semiprime $f$-algebra. Then
		\[
		\prod_{k=1}^{n}a_k=\prod_{k=1}^{n}\mathcal{M}_k(a_1,...,a_n)
		\]
		holds for all $a_1,...,a_n\in A$.
	\end{lemma}
	
	\begin{proof}
		Let $a_1,...,a_n\in A$, and consider the Archimedean $f$-subalgebra $C$ of $A$ generated by $\{a_1,...,a_n\}$. Suppose $\omega\colon C\to\mathbb{R}$ is a nonzero multiplicative homomorphism. Using the fact that $\{\omega(a_1),...,\omega(a_n)\}$ is a totally ordered set of real numbers together with Remark~\ref{R:totordsetssame} in the second identity below, we have
		\begin{align*}
			\omega\left(\prod_{k=1}^{n}a_k\right)&=\prod_{k=1}^{n}\omega(a_k)=\prod_{k=1}^{n}\mathcal{M}_k\Biggl(\omega(a_1),...,\omega(a_n)\Biggr)\\
			&=\prod_{k=1}^{n}\omega\Bigl(\mathcal{M}_k(a_1,...,a_n)\Bigr)=\omega\left(\prod_{k=1}^{n}\mathcal{M}_k(a_1,...,a_n)\right).
		\end{align*}
		Since the set of all nonzero multiplicative vector lattice homomorphisms $\omega\colon C\to\mathbb{R}$ separates the points of $C$ (see \cite[Corollary~2.7]{BusdPvR}), we have
		\[
		\prod_{k=1}^{n}a_k=\prod_{k=1}^{n}\mathcal{M}_k(a_1,...,a_n).
		\]
	\end{proof}
	
	We will use the following notation throughout the remainder of the paper.
	
	\begin{notation} For $m,n\in\mathbb{N}$ with $m\leq n$, $E$ uniformly complete, $Y$ a real separated convex bornologial space, a map $T\colon E^n\to Y$, $f_1,...,f_m\in E$, and $k_1,...,k_m\in\{0,1,2,...,n\}$ for which $\sum_{i=1}^{m}k_i=n$, we write
		\[
		T(f_1^{k_1}...f_m^{k_m})=T(\underbrace{f_1,\dots,f_1}_{k_1\ \text{copies}},\dots,\underbrace{f_m,\dots,f_m}_{k_m\ \text{copies}}).
		\]
	\end{notation}
	
	Using the technique employed in the proof of Kusraeva's Theorem~\ref{T:Kusa} along with Theorem~\ref{T:genorthosym}, we show that a bounded multilinear map is orthosymmetric if and only if it is total orderization invariant. As mentioned in the introduction, the equivalence $(i)\iff(ii)$ in Theorem~\ref{T:ortho} below is known for the special case when $E$ is a Banach lattice and $Y=\mathbb{R}$ (see \cite[Proposition~2]{BoyRySniga}).
	
	\begin{theorem}\label{T:ortho}
		Fix $n\in\mathbb{N}$. Let $E$ be uniformly complete and $Y$ a real separated convex bornological space. Suppose $T\colon E^n\to Y$ is a bounded $n$-linear map. The following are equivalent.
		
		\begin{itemize}
			\item[$(i)$] $T$ is orthosymmetric,
			\item[$(ii)$] $T$ is total orderization invariant, and
			\item[$(iii)$] $T$ is positively total orderization invariant.
		\end{itemize}
	\end{theorem}
	
	\begin{proof}
		$(i)\implies(ii)$ As previously, let $E^u$ denote the universal completion of $E$, and recall that $E^u$ is a semiprime $f$-algebra. Suppose that $f_1,\dots f_n\in E$. We claim that
		\[
		\textnormal{\textcircled{$n$}} (f_1,...,f_n)=\textnormal{\textcircled{$n$}} \Bigl(\mathcal{M}_1(f_1,...,f_n),\mathcal{M}_2(f_1,...,f_n),\dots,\mathcal{M}_n(f_1,...,f_n)\Bigr).
		\]
		Indeed, by \cite[Theorem~4.1]{BoBus}, there exists a uniformly complete vector sublattice $F$ of $E^u$ and a vector lattice isomorphism $i\colon E^{\textnormal{\textcircled{$n$}} }\to F$ such that both
		\[
		\prod_{k=1}^{n}x_k\in F
		\]
		and
		\[
		i\circ\textnormal{\textcircled{$n$}} (x_1,\dots,x_n)=\prod_{k=1}^{n}x_k
		\]
		hold for all $x_1,\dots,x_n\in E$.
		
		It follows from this fact and Lemma~\ref{L:mult} that 
		\begin{align*}
			\textnormal{\textcircled{$n$}} (f_1,...,f_n)&=i^{-1}\left(\prod_{k=1}^{n}f_k\right)\\
			&=i^{-1}\left(\prod_{k=1}^{n}\mathcal{M}_k(f_1,...,f_n)\right)\\
			&=\textnormal{\textcircled{$n$}} \Bigl(\mathcal{M}_1(f_1,...,f_n),\mathcal{M}_2(f_1,...,f_n),\dots,\mathcal{M}_n(f_1,...,f_n)\Bigr).
		\end{align*}
		
		By Theorem~\ref{T:KusaProofofThm4}, there exists a bounded linear operator $S\colon E^{\textnormal{\textcircled{$n$}} }\to Y$ for which
		\[
		T(x_1,\dots,x_n)=S\circ\textnormal{\textcircled{$n$}} (x_1,\dots,x_n)
		\]
		holds for every $x_1,\dots,x_n\in E$. It thus follows that
		\begin{align*}
			T(f_1,...,f_n)&=S\circ\textnormal{\textcircled{$n$}} (f_1,...,f_n)\\
			&=S\circ\textnormal{\textcircled{$n$}} \Bigl(\mathcal{M}_1(f_1,...,f_n),\mathcal{M}_2(f_1,...,f_n),\dots,\mathcal{M}_n(f_1,...,f_n)\Bigr)\\
			&=T\Bigl(\mathcal{M}_1(f_1,...,f_n),\mathcal{M}_2(f_1,...,f_n),\dots,\mathcal{M}_n(f_1,...,f_n)\Bigr).
		\end{align*}
		
		$(ii)\implies(iii)$ Trivial.
		
		$(iii)\implies(i)$ Suppose that 
		\[
		T(f_1,...,f_n)=T\Bigl(\mathcal{M}_1(f_1,...,f_n),\mathcal{M}_2(f_1,...,f_n),\dots,\mathcal{M}_n(f_1,...,f_n)\Bigr)
		\]
		holds for all $f_1,...,f_n\in E^+$. Note that the $n$-linearity of $T$ implies that $T(0,f_1,...,f_{n-1})=0$ holds for all $f_1,...,f_{n-1}\in E^+$. Therefore, it follows from Theorem~\ref{T:genorthosym} (with $L=E^+, \theta=0$, and $a=0$) that
		\begin{equation}\label{eq:posortho}
			T(f_1,...,f_n)=0\quad (f_1,...,f_n\in E^+\ \textnormal{with}\ f_i\wedge f_j=0\ \textnormal{for some}\ i,j\in[n]).
		\end{equation}
		
		Next we claim that $T$ is symmetric. Indeed, given a permutation $\sigma$ on $[n]$, for all $f_1,...,f_n\in E^+$ we have from Proposition~\ref{P:invarimpsym} that
		\begin{align*}
			T(f_1,...,f_n)=T\Bigl(f_{\sigma(1)},...,f_{\sigma(n)}\Bigr).
		\end{align*}
		Then using the $n$-linearity of $T$, we have for each $f_1,...,f_n\in E$ that
		\begin{align*}
			T&(f_1,...,f_n)=T(f_1^+-f_1^-,...,f_n^+-f_n^-)\\
			&=\sum_{k_1,...,k_n\in\{0,1\}}T\Bigl((f_1^+)^{k_1}(-f_1^-)^{1-k_1}\cdots(f_n^+)^{k_n}(-f_n^-)^{1-k_n}\Bigr)\\
			&=(-1)^{n-\sum_{i=1}^{n}k_i}\sum_{k_1,...,k_n\in\{0,1\}}T\Bigl((f_1^+)^{k_1}(f_1^-)^{1-k_1}\cdots(f_n^+)^{k_n}(f_n^-)^{1-k_n}\Bigr)\\
			&=(-1)^{n-\sum_{i=1}^{n}k_i}\sum_{k_1,...,k_n\in\{0,1\}}T\Bigl((f_{\sigma(1)}^+)^{k_{\sigma(1)}}(f_{\sigma(1)}^-)^{1-k_{\sigma(1)}}\cdots(f_{\sigma(n)}^+)^{k_{\sigma(n)}}(f_{\sigma(n)}^-)^{1-k_{\sigma(n)}}\Bigr)\\
			&=T(f_{\sigma(1)},...,f_{\sigma(n)}).
		\end{align*}
		Thus the map
		\[
		P_T(f)=T(\underbrace{f,\dots,f}_{n\ \text{copies}})\quad (f\in E)
		\]
		is an $n$-homogeneous polynomial with generating symmetric $n$-linear map $T$. We know from \cite[Corollary~1]{Kusa2} that $P_T$ is also bounded. For $f,g\in E^+$  with $f\wedge g=0$, the binomial theorem together with \eqref{eq:posortho} yield
		\begin{align*}
			P_T(f+g)&=P_T(f)+P_T(g)+\sum_{k=1}^{n}\binom{n}{k}T(f^{n-k}g^k)\\
			&=P_T(f)+P_T(g).
		\end{align*}
		Thus $P_T$ is positively orthogonally additive. Then $P_T$ is orthogonally additive by \cite[Theorem 2.3]{Sch}. By \cite[Lemma 4]{Kusa2}, $T$ is orthosymmetric.
	\end{proof}
	
	As gleaned from Corollaries~\ref{C:orthofuncal}\&\ref{C:steadyfuncal}, one advantage to studying total orderization invariance is that this theory leads to the discovery of new properties for certain functions. Along these lines, we obtain another characterization of bounded orthosymmetric multilinear maps.
	
	\begin{corollary}\label{C:troitsky}
		Let $E$ be uniformly complete, assume $Y$ is a real separated convex bornological space, and suppose $T\colon E^n\to Y$ is a bounded $n$-linear map. The following are equivalent.
		\begin{itemize}
			\item[$(i)$] $T$ is orthosymmetric,
			\item[$(ii)$] $T(f_1,...,f_n)=0$ whenever $f_1,...,f_n\in E$ satisfy $\bigwedge_{k=1}^n |f_{k}|=0$, and
			\item[$(iii)$] $T(f_1,...,f_n)=0$ whenever $f_1,...,f_n\in E^+$ satisfy $\bigwedge_{k=1}^nf_{k}=0$.
		\end{itemize}
	\end{corollary}
	
	\begin{proof}
		We first prove $(i)\implies(iii)$. To this end, assume $T$ is orthosymmetric and let $f_1,...,f_n\in E^+$ be such that $\mathcal{M}_1(f_1,...,f_n)=\bigwedge_{k=1}^nf_{k}=0$. From Theorem~\ref{T:ortho} we obtain
		\begin{align*}
			T(f_1,...,f_n)&=T\Bigl(\mathcal{M}_1(f_1,...,f_n),\mathcal{M}_2(f_1,...,f_n),\dots,\mathcal{M}_n(f_1,...,f_n)\Bigr)\\
			&=T\Bigl(0,\mathcal{M}_2(f_1,...,f_n),\dots,\mathcal{M}_n(f_1,...,f_n)\Bigr)\\
			&=0.
		\end{align*}
		
		Next we verify $(iii)\implies(ii)$. For this task assume $(iii)$ holds, and suppose $f_1,...,f_n\in E$ satisfy $\bigwedge_{k=1}^n|f_k|=0$. Recall that $0\leq f_k^+,f_k^-\leq|f_k|$ holds for each $k\in[n]$. It then follows from assumption $(iii)$ that for any $k_1,...,k_n\in\{0,1\}$ we have
		\[
		T\Bigl((f_1^+)^{k_1}(f_1^-)^{1-k_1}\cdots(f_n^+)^{k_n}(f_n^-)^{1-k_n}\Bigr)=0.
		\]
		We thus obtain
		\begin{align*}
			T(f_1,...,f_n)&=T(f_1^+-f_1^-,...,f_n^+-f_n^-)\\
			&=\sum_{k_1,...,k_n\in\{0,1\}}T\Bigl((f_1^+)^{k_1}(-f_1^-)^{1-k_1}\cdots(f_n^+)^{k_n}(-f_n^-)^{1-k_n}\Bigr)\\
			&=(-1)^{n-\sum_{i=1}^{n}k_i}\sum_{k_1,...,k_n\in\{0,1\}}T\Bigl((f_1^+)^{k_1}(f_1^-)^{1-k_1}\cdots(f_n^+)^{k_n}(f_n^-)^{1-k_n}\Bigr)\\
			&=0.
		\end{align*}
		
		We conclude the proof by noting the implication $(ii)\implies(i)$ is evident.
	\end{proof}
	
	\begin{remark}
		The logical equivalence $(i)\iff(ii)$ in Corollary~\ref{C:troitsky} is contained in \cite[Proposition~3.38]{Rob} by Roberts via a different proof. Maps that satisfy the conditions of Corollary~\ref{C:troitsky}$(ii)$ are referred to in \cite{Rob} as \textit{jointly orthosymmetric}.
	\end{remark}
	
	Next we turn our focus to orthogonally steady power sum polynomials. We obtain the following characterization of bounded orthogonally steady power sum polynomials using Theorem~\ref{T:genorthsteady}, Theorem~\ref{T:funcal}, and Theorem~\ref{T:Kusa}. For $E$ a Banach lattice and $Y=\mathbb{R}$, the equivalence $(i)\iff(ii)$ in Theorem~\ref{T:vals} is known and stated in terms of orthogonally additive polynomials in \cite[Remark (7) following Proposition~4]{BoyRySniga}.
	
	\begin{theorem}\label{T:vals}
		Fix $n,r\in\mathbb{N}$ with $r\geq 2$. Let $E$ be uniformly complete, and let $Y$ be a real separated convex bornological space. Suppose $S\colon E^r\to Y$ is a bounded power sum polynomial of degree $n$ in $r$ variables. The following are equivalent.
		\begin{itemize}
			\item[$(i)$] $S$ is orthogonally steady,
			\item[$(ii)$] $S$ is total orderization invariant, and
			\item[$(iii)$] $S$ is positively total orderization invariant. 
		\end{itemize}
	\end{theorem}
	
	\begin{proof}
		$(i)\implies(ii)$ Suppose that $S$ is orthogonally steady. As in Theorem~\ref{T:Kusa}, let $\mathfrak{S}$ be defined on $E$ by functional calculus. Let $f_1,...,f_r\in E$. Denote by $P$ the generating $n$-homogeneous polynomial for $S$. As noted in Lemma~\ref{L:steadyiffadditive}, $P$ is orthogonally additive. Moreover, it follows from the identity
		\[
		P(f)=S(\underbrace{0,\dots,0}_{r-1\ \text{copies}},f)\quad (f\in E)
		\]
		and the boundedness of $S$ that $P$ is bounded. Thus by Theorem~\ref{T:Kusa}, we have
		\[
		P\Bigl(\mathfrak{S}(f_1,...,f_r)\Bigr)=\sum_{k=1}^{r}P(f_k).
		\]
		From this result and Theorem~\ref{T:funcal} we get
		\begin{align*}
			S(f_1,...,f_r)&=\sum_{k=1}^{r}P(f_k)=P\Bigl(\mathfrak{S}(f_1,...,f_r)\Bigr)\\
			&=P\Biggl(\mathfrak{S}\Bigl(\mathcal{M}_1(f_1,...,f_r),\mathcal{M}_2(f_1,...,f_r),\dots,\mathcal{M}_r(f_1,...,f_r)\Bigr)\Biggr)\\
			&=\sum_{k=1}^{r}P\Bigl(\mathcal{M}_k(f_1,...,f_r)\Bigr)\\
			&=S\Bigl(\mathcal{M}_1(f_1,...,f_r),...,\mathcal{M}_r(f_1,...,f_r)\Bigr).
		\end{align*}
		
		$(ii)\implies(iii)$ Trivial.
		
		$(iii)\implies(i)$ Assume $S(f_1,...,f_r)=S\Bigl(\mathcal{M}_1(f_1,...,f_r),...,\mathcal{M}_r(f_1,...,f_r)\Bigr)$ holds for all $f_1,...,f_r\in E^+$. Since
		\[
		S(\underbrace{0,\dots,0}_{r-1\ \text{copies}},f)=P(f)
		\]
		holds for every $f\in E^+$, it follows from Theorem~\ref{T:genorthsteady} that
		\[
		S(f_1,...,f_r)=P\left(\bigvee_{k=1}^rf_k\right)
		\]
		whenever $f_1,...,f_r\in E^+$ satisfy $f_i\wedge f_j=0$ for all $i,j\in[r]$ with $i\neq j$.
		
		Consider $f_1,f_2\in E^+$ with $f_1\wedge f_2=0$. If $r\geq 3$, then additionally set $f_k=0$ for all integers $3\leq k\leq r$. Since $\{f_1,...,f_r\}$ is a pairwise disjoint subset of $E^+$, we get
		\[
		P(f_1+f_2)=P(f_1\vee f_2)=P\left(\bigvee_{k=1}^rf_k\right)=S(f_1,...,f_r)=\sum_{k=1}^{r}P(f_k)=P(f_1)+P(f_2).
		\]
		Thus $P$ is positively orthogonally additive. Hence $P$ is orthogonally additive by \cite[Theorem 2.3]{Sch}, and from Lemma~\ref{L:steadyiffadditive} we conclude that $S$ is orthogonally steady.
	\end{proof}

As an immediate corollary, we attain the following characterization for bounded orthogonally additive polynomials which generalizes \cite[Remark (7) following Proposition~4]{BoyRySniga}, a result which extends the notion of $n$-homogeneous polynomial valuations (see \cite[Section~2]{BusRob}) to any finite number of summands. A map $P\colon E\to V$, with $V$ a real vector space, is called a \textit{polynomial valuation} if $P$ is an $n$-homogeneous polynomial for some $n\in\mathbb{N}$ and a valuation, meaning that $P(f)+P(g)=P(f\wedge g)+P(f\vee g)$ holds for every $f,g\in E$.
	
	\begin{corollary}\label{C:vals}
		Fix $n,r\in\mathbb{N}$ with $r\geq 2$. Let $E$ be uniformly complete, and let $Y$ be a real separated convex bornological space. Suppose $P\colon E\to Y$ is a bounded $n$-homogeneous polynomial. The following are equivalent.
		\begin{itemize}
			\item[$(i)$] $P$ is orthogonally additive,
			\item[$(ii)$] $\sum_{k=1}^{r}P(f_k)=\sum_{k=1}^{r}P\Bigl(\mathcal{M}_k(f_1,...,f_r)\Bigr)$ holds for every $f_1,...,f_r\in E$, and
			\item[$(iii)$] $\sum_{k=1}^{r}P(f_k)=\sum_{k=1}^{r}P\Bigl(\mathcal{M}_k(f_1,...,f_r)\Bigr)$ holds for every $f_1,...,f_r\in E^+$. 
		\end{itemize}
	\end{corollary}
	
	In our final note, we know from e.g. \cite[Theorem~2.3]{Sch} that, in the present context, a symmetric $n$-linear map $T\colon E^n\to Y$ is orthosymmetric if and only if the $n$-homogeneous polynomial $P_T$ it generates is orthogonally additive. By definition, $P_T$ generates a unique power sum polynomial $S_{P_T}$ of degree $n$ in $n$ variables. From Lemma~\ref{L:steadyiffadditive} we know for $n\geq 2$ that $P_T$ is orthogonally additive if and only if $S_{P_T}$ is orthogonally steady. Hence $T$ is orthosymmetric if and only if $S_{P_T}$ is orthogonally steady. By using Theorems~\ref{T:ortho}\&\ref{T:vals}, this logical equivalence can be expressed completely in terms of total orderization invariance.
	
	\begin{corollary}
		Fix $n\in\mathbb{N}$ with $n\geq 2$. Let $E$ be uniformly complete, and let $Y$ be a real separated convex bornological space. Consider a bounded symmetric $n$-linear map $T\colon E^n\to Y$. Let $P_T$ be the $n$-homogeneous polynomial generated by $T$, and denote the power sum polynomial of degree $n$ in $n$ variables which $P_T$ generates by $S_{P_T}$. The following are equivalent.
		\begin{itemize}
			\item[$(i)$] $T$ is total orderization invariant,
			\item[$(ii)$] $T$ is positively total orderization invariant,
			\item[$(iii)$] $S_{P_T}$ is total orderization invariant, and
			\item[$(iv)$] $S_{P_T}$ is positively total orderization invariant.
		\end{itemize}
	\end{corollary}


\begin{thebibliography}{1}
		\bibitem{AB} Aliprantis, C.D., Burkinshaw, O.: Positive Operators. Academic Press, Orlando (1985)
		
		\bibitem{AzBoBus} Azzouzi, Y., Boulabiar, K., Buskes, G.: The de Schipper formula and squares of Riesz spaces. Indag. Math. (N.S.) 17, no. 4, 479-496 (2006)
		
		\bibitem{BerHui2} Bernau, S.J., Huijsmans, C.B.: Almost $f$-algebras and $d$-algebras. Math. Proc. Cambridge Philos. Soc. 107, no. 2, 287-308 (1990)
		
		\bibitem{Birk} Birkhoff, G.: Lattice Theory, Third Ed. American Mathematical Society Colloquium Publications, Vol. 25, American Mathematical Society, Providence (1979)
		
		\bibitem{BoBus} Boulabiar, K., Buskes, G.: Vector lattice powers: $f$-algebras and functional calculus. Comm. Algebra 34, no. 4, 1435-1442 (2006)
		
		\bibitem{BoyRySniga} Boyd, C., Ryan, R., Snigireva, N.: A Nakano carrier theorem for polynomials. arXiv (2021)
		
		\bibitem{BusdPvR} Buskes, G., de~Pagter, B., van Rooij, A.: Functional calculus on Riesz spaces. Indag. Math. (N.S.) 2, no. 4, 423-436 (1991)
		
		\bibitem{BusRob} Buskes, G., Roberts, S.: Valuations: from orthogonal additivity to orthosymmetry. Indag. Math. (N.S.) 31, no. 5, 758-766 (2020)
		
		\bibitem{BusSch2} Buskes, G., Schwanke, C.: Complex vector lattices via functional completions. J. Math. Anal. Appl. 434, no. 2, 1762-1778 (2016)
		
		\bibitem{BusSch} Buskes, G., Schwanke, C.: Functional completions of Archimedean vector lattices. Algebra Universalis 76, no. 1, 53-69 (2016)
		
		\bibitem{Kus3} Kusraev, A.G.: On some properties of orthosymmetric bilinear operators. Vladikavkaz. Mat. Zh. 10, no. 3, 29-33 (2008)
		
		\bibitem{Kusa2} Kusraeva, Z.A.: On the representation of orthogonally additive polynomials. Sibirsk. Mat. Zh. 52, no. 2, 315-325 (2011)
		
		\bibitem{Kusa} Kusraeva, Z.A.: Homogeneous polynomials, power means and geometric means in vector lattices. Vladikavkaz. Mat. Zh. 16, no. 4, 49-53 (2014)
		
		\bibitem{LuxZan1} Luxemburg, W.A.J., Zaanen, A.C.: Riesz Spaces Vol. I. North-Holland Publishing Co., Amsterdam-London, American Elsevier Publishing
		Co., New York (1971)
		
		\bibitem{Rob} Roberts, S.C.: Orthosymmetric Maps and Polynomial Valuations. ProQuest LLC, Ann Arbor, MI, Thesis (Ph.D.)-The University of
		Mississippi (2017)
		
		\bibitem{Sch} Schwanke, C.: Some notes on orthogonally additive polynomials. Quaest.
		Math. (2021). https://doi.org/10.2989/16073606.2021.1953631 
		
		\bibitem{Zan2} Zaanen, A.C.: Riesz Spaces II. North-Holland Mathematical Library,
		vol. 30, North-Holland Publishing Co., Amsterdam (1983)
	\end{thebibliography}
\end{document}